\documentclass[12pt,a4paper]{amsart}
\usepackage{amssymb,amsxtra}
\usepackage[english,french]{babel}

\newtheorem{thm}{Theorem}[section]
\newtheorem{lem}[thm]{Lemma}

\theoremstyle{definition}

\theoremstyle{remark}

\theoremstyle{plain}

\theoremstyle{remark}

\newtheorem*{example}{Example}
\numberwithin{equation}{section}

\begin{document}

\title{On a conjecture concerning  enumeration of $ 2n\times k $ n-times  Persymmetric  Matrices over $\mathbb{F}_{2} $ by rank}
\author{Jorgen~Cherly}
\address{D\'epartement de Math\'ematiques, Universit\'e de
    Brest, 29238 Brest cedex~3, France}
\email{Jorgen.Cherly@univ-brest.fr}
\email{andersen69@wanadoo.fr}
\maketitle 

\begin{abstract}
Dans cet article nous annon\c{c}ons une conjecture concernant l'\' enum\' eration de $2n\times k$  n- fois matrices persym\' etriques sur  $\mathbb{F}_{2}$ par le rang.
 \end{abstract}

\selectlanguage{english}

\begin{abstract}
In this paper we announce a conjecture concerning enumeration of $2n\times k $ n-times  persymmetric  matrices over $\mathbb{F}_{2}$ by rank.
 \end{abstract}
\maketitle 
\newpage
\tableofcontents
\newpage

   \section{Introduction.}
  \label{sec 1}  
  In this paper we propose a formula to compute  the number $ \Gamma_{i}^{\left[2\atop{\vdots \atop 2}\right]\times k}$  of rank i  $2n\times k $
   n-times persymmetric matrices over $\mathbb{F}_{2}$ of the below form for $0\leqslant i\leqslant \inf(2n,k) $  \\
  \begin{equation}
  \label{eq 1.1}
   \left (  \begin{array} {cccccccc}
\alpha  _{1}^{(1)} & \alpha  _{2}^{(1)}  &   \alpha_{3}^{(1)} &   \alpha_{4}^{(1)} &   \alpha_{5}^{(1)} &  \alpha_{6}^{(1)}  & \ldots  &  \alpha_{k}^{(1)} \\
\alpha  _{2}^{(1)} & \alpha  _{3}^{(1)}  &   \alpha_{4}^{(1)} &   \alpha_{5}^{(1)} &   \alpha_{6}^{(1)} &  \alpha_{7}^{(1)} & \ldots  &  \alpha_{k+1}^{(1)} \\ 
\hline \\
\alpha  _{1}^{(2)} & \alpha  _{2}^{(2)}  &   \alpha_{3}^{(2)} &   \alpha_{4}^{(2)} &   \alpha_{5}^{(2)} &  \alpha_{6}^{(2)} & \ldots   &  \alpha_{k}^{(2)} \\
\alpha  _{2}^{(2)} & \alpha  _{3}^{(2)}  &   \alpha_{4}^{(2)} &   \alpha_{5}^{(2)}&   \alpha_{6}^{(2)} &  \alpha_{7}^{(2)}  & \ldots  &  \alpha_{k+1}^{(2)} \\ 
\hline\\
\alpha  _{1}^{(3)} & \alpha  _{2}^{(3)}  &   \alpha_{3}^{(3)}  &   \alpha_{4}^{(3)} &   \alpha_{5}^{(3)} &  \alpha_{6}^{(3)} & \ldots  &  \alpha_{k}^{(3)} \\
\alpha  _{2}^{(3)} & \alpha  _{3}^{(3)}  &   \alpha_{4}^{(3)}&   \alpha_{5}^{(3)} &   \alpha_{6}^{(3)}  &  \alpha_{7}^{(3)} & \ldots  &  \alpha_{k+1}^{(3)} \\ 
\hline \\
\vdots & \vdots & \vdots  & \vdots  & \vdots & \vdots  & \vdots & \vdots \\
\hline \\
\alpha  _{1}^{(n)} & \alpha  _{2}^{(n)}  &   \alpha_{3}^{(n)} &   \alpha_{4}^{(n)} &   \alpha_{5}^{(n)}  &  \alpha_{6}^{(n)} & \ldots  &  \alpha_{k}^{(n)} \\
\alpha  _{2}^{(n)} & \alpha  _{3}^{(n)}  &   \alpha_{4}^{(n)}&   \alpha_{5}^{(n)} &   \alpha_{6}^{(n)}  &  \alpha_{7}^{(n)} & \ldots  &  \alpha_{k+1}^{(n)} \\ 
\end{array} \right ).  
\end{equation} 
We remark that the  n-times  persymmetric matrice above is the most simple case of all  n-times  persymmetric matrices over $\mathbb{F}_{2}$\\
More precisely we postulate that : \\

    \begin{equation}
\label{eq 1.2}  
   \displaystyle \Gamma_{i}^{\left[2\atop{\vdots \atop 2}\right]\times k}  = (2^{i+1}-1)\cdot 2^{in} + \sum_{j=0}^{i-1}a_{j}^{(i)}(k)\cdot2^{jn}
      \end{equation}
     \begin{equation*}  
  \displaystyle  = (2^{n}-1)(2^{n}-2)\ldots (2^{n}-2^{E(\frac{i-1}{2})})\cdot \big[(2^{i+1}-1)\cdot 2^{(i-1-E(\frac {i-1}{2}))n}
 + \sum_{j=0}^{i-2-E(\frac{i-1}{2})} \alpha_{j}^{(i)}(k)\cdot 2^{jn}\big]     \quad \text{if} \quad  i < \inf (2n,k).  \\   
     \end{equation*}   
   \begin{equation}
\label{eq 1.3}  
   \displaystyle \Gamma_{i}^{\left[2\atop{\vdots \atop 2}\right]\times k}  =  \sum_{j=0}^{E(\frac{i}{2})}b_{j}^{(i)}(n)\cdot2^{jk}
    \quad \text{if} \quad  i \leqslant \inf (2n,k).  
     \end{equation}
      \begin{equation}
\label{eq 1.4}  
   \displaystyle  \Gamma_{2n}^{\left[2\atop{\vdots \atop 2}\right]\times k}= 2^{n}\prod_{j=1}^{j=n}(2^k-2^{2n-j})  \quad \text{if} \quad  k\geqslant 2n. \\
   \end{equation}
        \begin{equation}
\label{eq 1.5}  
 \displaystyle  \Gamma_{k}^{\left[2\atop{\vdots \atop 2}\right]\times k}= 2^{n(k+1)}-\sum_{i=0}^{i=k-1} \Gamma_{i}^{\left[2\atop{\vdots \atop 2}\right]\times k}  \quad \text{if} \quad k< 2n.
   \end{equation}\\

        \begin{equation}
\label{eq 1.6}  
 \displaystyle a_{i-1}^{(i)}(k)= a(i)\cdot2^k+b(i)\text{ where }   a(i),b(i) \in \mathbb{Q}.
   \end{equation}

   \section{Study concerning the formula \eqref{eq 1.2} }
  \label{sec 2}  
\subsection{Justification of the formula  \eqref{eq 1.2} in the case n=1}
\label{subsec 2.1}
We consider the following   $2\times k $  persymmetric  matrice over $\mathbb{F}_{2}$.\\
   \begin{equation}
  \label{eq 2.1}
   \left (  \begin{array} {cccccccc}
\alpha  _{1}^{(1)} & \alpha  _{2}^{(1)}  &   \alpha_{3}^{(1)} &   \alpha_{4}^{(1)} &   \alpha_{5}^{(1)} &  \alpha_{6}^{(1)}  & \ldots  &  \alpha_{k}^{(1)} \\
\alpha  _{2}^{(1)} & \alpha  _{3}^{(1)}  &   \alpha_{4}^{(1)} &   \alpha_{5}^{(1)} &   \alpha_{6}^{(1)} &  \alpha_{7}^{(1)} & \ldots  &  \alpha_{k+1}^{(1)} 
\end{array} \right ).  
\end{equation} 
We obtain from Daykin [3] or Cherly [4] \\
\begin{equation}
\label{eq 2.2}
   \Gamma_{i}^{2\times k}=   \begin{cases}
 1 & \text{if  } i = 0,  \\  
3 & \text{if  } i = 1,   \\
2^{k+1}-4 & \text{if  } i = 2.\\
 \end{cases}
   \end{equation} 
From the formula  \eqref{eq 1.2} with n=1, i=1 and $ k> 1 $  using \eqref{eq 2.2} we obtain :\\
\begin{equation*}
 \displaystyle  \Gamma_{1}^{2\times k} =  (2^{2}-1)\cdot 2 + a_{0}^{(1)}(k) =   6+a_{0}^{(1)}(k) = 3,\quad  a_{0}^{(1)}(k) = -3.
   \end{equation*}
\subsection{Justification of the formula  \eqref{eq 1.2} in the case n=2}
\label{subsec 2.2}
We consider the following   $4\times k $ double  persymmetric  matrice over $\mathbb{F}_{2}$.\\
  \begin{equation*}
 \left (  \begin{array} {cccccccc}
\alpha  _{1}^{(1)} & \alpha  _{2}^{(1)}  &   \alpha_{3}^{(1)} &   \alpha_{4}^{(1)} &   \alpha_{5}^{(1)} &  \alpha_{6}^{(1)}  & \ldots  &  \alpha_{k}^{(1)} \\
\alpha  _{2}^{(1)} & \alpha  _{3}^{(1)}  &   \alpha_{4}^{(1)} &   \alpha_{5}^{(1)} &   \alpha_{6}^{(1)} &  \alpha_{7}^{(1)} & \ldots  &  \alpha_{k+1}^{(1)} \\ 
\hline \\
\alpha  _{1}^{(2)} & \alpha  _{2}^{(2)}  &   \alpha_{3}^{(2)} &   \alpha_{4}^{(2)} &   \alpha_{5}^{(2)} &  \alpha_{6}^{(2)} & \ldots   &  \alpha_{k}^{(2)} \\
\alpha  _{2}^{(2)} & \alpha  _{3}^{(2)}  &   \alpha_{4}^{(2)} &   \alpha_{5}^{(2)}&   \alpha_{6}^{(2)} &  \alpha_{7}^{(2)}  & \ldots  &  \alpha_{k+1}^{(2)} \\ 
\end{array} \right ).  
\end{equation*} 
We obtain from Cherly [5] \\
\begin{equation}
\label{eq 2.3}
\Gamma_{i}^{\left[2\atop 2\right]\times k} =   \begin{cases}
 1 & \text{if  } i = 0,  \\  
9 & \text{if  } i = 1,   \\
3\cdot 2^{k+1}+30 & \text{if  } i = 2,\\
21\cdot 2^{k+1}-168 & \text{if  } i = 3,\\
2^{2k+2} -3\cdot 2^{k+4}+128 & \text{if  } i = 4.
 \end{cases}
   \end{equation} 
From the formula  \eqref{eq 1.2} with n=2, i=1 and $ k> 1 $  using \eqref{eq 2.3} we obtain :\\
\begin{equation*}
 \displaystyle  \Gamma_{1}^{\left[2\atop 2\right]\times k}  =  (2^{2}-1)\cdot 2^2 + a_{0}^{(1)}(k) =   12+a_{0}^{(1)}(k) = 12-3=9.
   \end{equation*}
From the formula  \eqref{eq 1.2} with n=2, i=2 and $ k> 2 $  using \eqref{eq 2.3}  we obtain :\\
\begin{equation*}
 \displaystyle  \Gamma_{2}^{\left[2\atop 2\right]\times k}  =  (2^{3}-1)\cdot 2^4 + a_{0}^{(2)}(k) +4\cdot a_{1}^{(2)}(k)  =   3\cdot 2^{k+1}+30. 
   \end{equation*}
From \eqref{eq 1.2} we deduce :\\
\begin{equation}
\label{eq 2.4}
 a_{j}^{(2)}(k)    =   \begin{cases}
-2^{k+1}+18   & \text{if  } j = 0,  \\  
2^{k+1}-25  & \text{if  } j = 1.  \\
 \end{cases}
 \end{equation} 
From the formula  \eqref{eq 1.2} with n=2, i=3 and $ k> 3 $  using \eqref{eq 2.3}  we obtain :\\
\begin{equation*}
 \displaystyle  \Gamma_{3}^{\left[2\atop 2\right]\times k}  =  (2^{4}-1)\cdot 2^6 + a_{0}^{(3)}(k) +4\cdot a_{1}^{(3)}(k) + 2^4\cdot a_{2}^{(3)}(k)=   21\cdot 2^{k+1}- 168.
   \end{equation*}
From \eqref{eq 1.2} we deduce :\\
\begin{equation}
\label{eq 2.5}
 a_{j}^{(3)}(k)    =   \begin{cases}
14\cdot2^{k}-176   & \text{if  } j = 0,  \\  
-21\cdot2^{k}+294  & \text{if  } j = 1,   \\
7\cdot2^{k}-133 & \text{if  } j = 2.
 \end{cases}
   \end{equation} 

\subsection{Justification of the formula  \eqref{eq 1.2} in the case n=3}
\label{subsec 2.3}
We consider the following   $6\times k $ triple persymmetric  matrice over $\mathbb{F}_{2}$\\
  \begin{equation*}
 \left (  \begin{array} {cccccccc}
\alpha  _{1}^{(1)} & \alpha  _{2}^{(1)}  &   \alpha_{3}^{(1)} &   \alpha_{4}^{(1)} &   \alpha_{5}^{(1)} &  \alpha_{6}^{(1)}  & \ldots  &  \alpha_{k}^{(1)} \\
\alpha  _{2}^{(1)} & \alpha  _{3}^{(1)}  &   \alpha_{4}^{(1)} &   \alpha_{5}^{(1)} &   \alpha_{6}^{(1)} &  \alpha_{7}^{(1)} & \ldots  &  \alpha_{k+1}^{(1)} \\ 
\hline \\
\alpha  _{1}^{(2)} & \alpha  _{2}^{(2)}  &   \alpha_{3}^{(2)} &   \alpha_{4}^{(2)} &   \alpha_{5}^{(2)} &  \alpha_{6}^{(2)} & \ldots   &  \alpha_{k}^{(2)} \\
\alpha  _{2}^{(2)} & \alpha  _{3}^{(2)}  &   \alpha_{4}^{(2)} &   \alpha_{5}^{(2)}&   \alpha_{6}^{(2)} &  \alpha_{7}^{(2)}  & \ldots  &  \alpha_{k+1}^{(2)} \\ 
\hline\\
\alpha  _{1}^{(3)} & \alpha  _{2}^{(3)}  &   \alpha_{3}^{(3)}  &   \alpha_{4}^{(3)} &   \alpha_{5}^{(3)} &  \alpha_{6}^{(3)} & \ldots  &  \alpha_{k}^{(3)} \\
\alpha  _{2}^{(3)} & \alpha  _{3}^{(3)}  &   \alpha_{4}^{(3)}&   \alpha_{5}^{(3)} &   \alpha_{6}^{(3)}  &  \alpha_{7}^{(3)} & \ldots  &  \alpha_{k+1}^{(3)} \end{array} \right ).  
\end{equation*} 
We obtain from Cherly [6] \\
\begin{equation}
\label{eq 2.6}
\Gamma_{i}^{\left[2\atop {2\atop 2}\right]\times k} =   \begin{cases}
1  &\text{if  }  i = 0, \\
21   &\text{if  }  i = 1, \\
7\cdot2^{k+ 1} + 266  & \text{if   } i = 2, \\
147\cdot2^{k+ 1} + 1344   & \text{if   } i = 3, \\
7\cdot2^{2k+2} + 651\cdot2^{k+ 2}  - 22624 & \text{if   } i = 4, \\
105\cdot2^{2k+2} - 315\cdot2^{k+ 5} + 53760   & \text{if   } i = 5, \\
2^{3k+3} - 7\cdot2^{2k+6} + 7\cdot2^{k+ 10} - 32768  & \text{if  } i = 6,\;k\geq 6.
 \end{cases}
   \end{equation} 
From the formula  \eqref{eq 1.2} with n=3, i=1 and $ k> 1 $  using \eqref{eq 2.4} we obtain :\\
\begin{equation*}
 \displaystyle  \Gamma_{1}^{\left[2\atop{ 2\atop 2}\right]\times k}  =  (2^{2}-1)\cdot 2^3 + a_{0}^{(1)}(k) =   24+a_{0}^{(1)}(k) = 24-3=21.
   \end{equation*}
  
  From the formula  \eqref{eq 1.2} with n=3, i=2 and $ k> 2 $  using \eqref{eq 2.4}  we obtain :\\
\begin{equation*}
 \displaystyle  \Gamma_{2}^{\left[2\atop {2\atop 2}\right]\times k}  =  (2^{3}-1)\cdot 2^6 + a_{0}^{(2)}(k) +8\cdot a_{1}^{(2)}(k)  =  
 7\cdot 64 +(-2^{k+1}+18) +8\cdot (2^{k+1}-25 ) =7\cdot2^{k+ 1} + 266. 
   \end{equation*}

From the formula  \eqref{eq 1.2} with n=3, i=3 and $ k> 3 $  using \eqref{eq 2.5}  we obtain :\\
\begin{eqnarray*}
 \displaystyle  \Gamma_{3}^{\left[2\atop {2\atop 2}\right]\times k} & = & (2^{4}-1)\cdot 2^9 + a_{0}^{(3)}(k) +8\cdot a_{1}^{(3)}(k) + 2^6\cdot a_{2}^{(3)}(k)\\
   &= & 15\cdot 512+  (14\cdot2^{k}-176 ) +8\cdot( -21\cdot2^{k}+294) +64\cdot(7\cdot2^{k}-133)\\
 &  = &147\cdot 2^{k+1}+1344.
\end{eqnarray*}

From the formula  \eqref{eq 1.2} with n=3, i=4 and $ k> 4 $ using \eqref{eq 2.6}  we obtain :\\
\begin{eqnarray*}
 \displaystyle  \Gamma_{4}^{\left[2\atop {2\atop 2}\right]\times k} & = & (2^{5}-1)\cdot 2^{12} + a_{0}^{(4)}(k) +2^3\cdot a_{1}^{(4)}(k) + 2^6\cdot a_{2}^{(4)}(k)+  2^{9}\cdot a_{3}^{(4)}(k) \\
 & = & 7\cdot2^{2k+2} + 651\cdot2^{k+ 2}  - 22624 \\
  &= & 31\cdot 2^{12}+  \frac{2^{2k+2}-117\cdot 2^{k+2}+9440}{3}  +2^3\cdot(-2^{2k+1}+269\cdot2^{k}-5744 )\\
& + & 2^6\cdot( \frac{2^{2k+2}-783\cdot2^{k}+19028}{6})+2^{9}\cdot( \frac{35\cdot2^{k}-1210}{2}  ).
  \end{eqnarray*}
  Thus we have :\\
  
  \begin{equation}
\label{eq 2.7}
 a_{j}^{(4)}(k)    =   \begin{cases}
  \frac{2^{2k+2}-117\cdot 2^{k+2}+9440}{3}  & \text{if  } j = 0,  \\  
 -2^{2k+1}+269\cdot2^{k}-5744  & \text{if  } j = 1,   \\
  \frac{2^{2k+2}-783\cdot2^{k}+19028}{6} & \text{if  } j = 2,\\
   \frac{35\cdot2^{k}-1210}{2}  & \text{if  } j = 3.
 \end{cases}
   \end{equation} 
  From the formula  \eqref{eq 1.2} with n=3, i=5 and $ k> 5 $ using \eqref{eq 2.6}  we obtain :\\
\begin{eqnarray*}
 \displaystyle  \Gamma_{5}^{\left[2\atop {2\atop 2}\right]\times k} & =& (2^{6}-1)\cdot 2^{15} + \sum_{i=0}^{4}a_{i}^{(5)}(k)\cdot2^{3i} \\
&  = & (2^{3}-1)(2^{3}-2) (2^{3}-2^{2})\cdot \big[(2^{6}-1)\cdot 2^{6} 
 +  \alpha_{0}^{(5)}(k)+ \alpha_{1}^{(5)}(k)\cdot 2^{3}\big]      \\  
 &=&168\cdot \big[4032 +  \alpha_{0}^{(5)}(k)+ \alpha_{1}^{(5)}(k)\cdot 2^{3}\big]         \\  
 &=& 63\cdot2^{15}+2^{12}\cdot\big[ \alpha_{1}^{(5)}(k)-441\big]+2^9\cdot\big[ \alpha_{0}^{(5)}(k)-7\cdot \alpha_{1}^{(5)}(k)+882\big]\\
 &+&2^6\cdot\big[-7\cdot \alpha_{0}^{(5)}(k)+14\cdot \alpha_{1}^{(5)}(k)-1008\big]  \\
 &+& 2^3\cdot\big[14\cdot \alpha_{0}^{(5)}(k)-8\cdot \alpha_{1}^{(5)}(k)\big]-8\cdot \alpha_{0}^{(5)}(k) \\
 & = &105\cdot 2^{2k+2} - 315\cdot 2^{k+ 5} + 53760.  
  \end{eqnarray*}
We then obtain :\\
\begin{equation}
\label{eq 2.8}
 \alpha_{0}^{(5)}(k)+8\cdot \alpha_{1}^{(5)}(k)=\frac{5}{2}\cdot2^{2k} -60\cdot2^{k}  - 3712.
\end{equation}
and \\
  \begin{equation}
\label{eq 2.9}
 a_{i}^{(5)}(k)    =   \begin{cases}
 -8\cdot \alpha_{0}^{(5)}(k)   & \text{if  } i = 0,  \\  
 14\cdot \alpha_{0}^{(5)}(k)-8\cdot \alpha_{1}^{(5)}(k)  & \text{if  } i = 1,    \\
 -7\cdot \alpha_{0}^{(5)}(k)+14\cdot \alpha_{1}^{(5)}(k)-1008 & \text{if  } i = 2,  \\
   \alpha_{0}^{(5)}(k)-7\cdot \alpha_{1}^{(5)}(k)+882  & \text{if  } i = 3, \\
   \alpha_{1}^{(5)}(k)-441  & \text{if  } i= 4. 
 \end{cases}
   \end{equation}

We shall need the following result from Theorem 5.2 in Cherly [12] \\
\begin{equation}
\label{eq 2.10}
 \Gamma_{5}^{\left[2\atop {2\atop {2\atop2 }}\right]\times k}= 6300 \cdot (2^{2k}+100 \cdot2^{k} -1856).
\end{equation}
   
  From the formula  \eqref{eq 1.2} with n=4, i=5 and $ k> 5 $ using \eqref{eq 2.10}  we obtain :\\
\begin{eqnarray*}
 \displaystyle  \Gamma_{5}^{\left[2\atop {2\atop{ 2\atop 2}}\right]\times k} & = & (2^{4}-1)(2^{4}-2) (2^{4}-2^{2})\cdot \big[(2^{6}-1)\cdot 2^{8}
 +  \alpha_{0}^{(5)}(k)+ \alpha_{1}^{(5)}(k)\cdot 2^{4}\big]      \\  
 &=&2520\cdot \big[16128 +  \alpha_{0}^{(5)}(k)+ \alpha_{1}^{(5)}(k)\cdot 2^{4}\big]      \\  
 & = & 6300 \cdot (2^{2k}+100 \cdot2^{k} -1856).
  \end{eqnarray*}
We then obtain :\\
\begin{equation}
\label{eq 2.11}
 \alpha_{0}^{(5)}(k)+16\cdot \alpha_{1}^{(5)}(k)=\frac{5}{2}\cdot2^{2k} +250\cdot2^{k}  - 20768.
\end{equation}
From \eqref{eq 2.8} and \eqref{eq 2.11} we get :\\
  \begin{equation}
\label{eq 2.12}
  \alpha_{j}^{(5)}(k)   =   \begin{cases}
  \frac{5}{2}\cdot2^{2k}-370\cdot2^k+13344  & \text{if  } j = 0,  \\  
 \frac{155}{4}\cdot2^k-2132 & \text{if  } j = 1.  \\
 \end{cases}
   \end{equation} 
We then deduce from \eqref{eq 2.9} and \eqref{eq 2.12} \\
  \begin{equation}
\label{eq 2.13}
 a _{i}^{(5)}(k)   =   \begin{cases}
  -20\cdot2^{2k}+2960\cdot2^{k}-106752     & \text{if  } i = 0,  \\  
 35\cdot2^{2k}-5490\cdot2^{k}+203872  & \text{if  } i= 1,   \\
 \frac{1}{2}\cdot(-35\cdot2^{2k}+6265\cdot2^{k}-247520) & \text{if  } i= 2,   \\
  \frac{5}{2}\cdot2^{2k}-\frac{2565}{4}\cdot2^{k}+29150   & \text{if  } i= 3,   \\
   \frac{155}{4}\cdot2^{k}-2573  & \text{if  } i= 4.  \\
  \end{cases}
   \end{equation} 

   \section{Study concerning the formula \eqref{eq 1.4} }
  \label{sec 3}  
\subsection{Justification of the formula  \eqref{eq 1.4} in the case n=1}
\label{subsec 3.1}
From the formula  \eqref{eq 1.4} with n=1, i=2 and $ k\geqslant 2 $  using \eqref{eq 2.2} we obtain :\\
\begin{equation*}
 \displaystyle  \Gamma_{2}^{2\times k} =  2^{n}\prod_{j=1}^{j=n}(2^k-2^{2n-j}) = 2\cdot (2^k-2)=2^{k+1}-4.
   \end{equation*}
\subsection{Justification of the formula  \eqref{eq 1.4} in the case n=2}
\label{subsec 3.2}
 From the formula  \eqref{eq 1.4} with n=2, i=4 and $ k\geqslant 4 $  using \eqref{eq 2.3} we obtain :\\
\begin{equation*}
 \displaystyle  \Gamma_{4}^{\left[2\atop 2\right]\times k}  =  2^{n}\prod_{j=1}^{j=n}(2^k-2^{2n-j}) = 2^2\cdot(2^k-2^3)(2^k-2^2)=2^{2k+2} -3\cdot 2^{k+4}+128. 
    \end{equation*}
\subsection{Justification of the formula  \eqref{eq 1.4} in the case n=3}
\label{subsec 2.3}
 From the formula  \eqref{eq 1.4} with n=3, i=6 and $ k\geqslant 6 $  using \eqref{eq 2.6} we obtain :\\
\begin{equation*}
 \displaystyle  \Gamma_{6}^{\left[2\atop {2\atop 2}\right]\times k}  =  2^{n}\prod_{j=1}^{j=n}(2^k-2^{2n-j}) = 2^3\cdot(2^k-2^5)(2^k-2^4)(2^k-2^3)= 
 2^{3k+3} - 7\cdot2^{2k+6} + 7\cdot2^{k+ 10} - 32768. 
    \end{equation*}

   \section{Study concerning the formula \eqref{eq 1.3} }
  \label{sec 4}  
\subsection{Justification of the formula  \eqref{eq 1.3} in the case n=1}
\label{subsec 4.1}
 From the formula  \eqref{eq 1.3} with n=1 and  using \eqref{eq 2.2} we obtain :\\
 \begin{equation*}
\begin{pmatrix}                                                 
b_{0}^{(1)}(1) & 0\\
b_{0}^{(2)}(1) &  b_{1}^{(2)}(1)
\end{pmatrix} =
\begin{pmatrix}
3 & 0 \\
-4 &  2
\end{pmatrix}.
\end{equation*}
\subsection{Justification of the formula  \eqref{eq 1.3} in the case n=2}
\label{subsec 4.2}
From the formula  \eqref{eq 1.3} with n=2 and  using \eqref{eq 2.3} we obtain :\\
 \begin{equation*}
\begin{pmatrix}                                                 
b_{0}^{(1)}(2) & 0 & 0 \\
b_{0}^{(2)}(2) &  b_{1}^{(2)}(2) & 0 \\
b_{0}^{(3)}(2) &  b_{1}^{(3)}(2) & 0 \\
b_{0}^{(4)}(2) &  b_{1}^{(4)}(2) &   b_{2}^{(4)}(2) 
\end{pmatrix} =
\begin{pmatrix}
9 & 0 & 0 \\
30 & 6 & 0 \\
-168 &  42 & 0 \\
128 &  -48 &   4
\end{pmatrix}.
\end{equation*}
\subsection{Justification of the formula  \eqref{eq 1.3} in the case n=3}
\label{subsec 4.3}
From the formula  \eqref{eq 1.3} with n=3 and  using \eqref{eq 2.6} we obtain :\\
 \begin{equation*}
\begin{pmatrix}                                                 
b_{0}^{(1)}(3) & 0 & 0 & 0\\
b_{0}^{(2)}(3) &  b_{1}^{(2)}(3) & 0 & 0\\
b_{0}^{(3)}(3) &  b_{1}^{(3)}(3) & 0 & 0 \\
b_{0}^{(4)}(3) &  b_{1}^{(4)}(3) &   b_{2}^{(4)}(3) & 0\\
b_{0}^{(5)}(3) &  b_{1}^{(5)}(3) &   b_{2}^{(5)}(3) & 0\\
b_{0}^{(6)}(3) &  b_{1}^{(6)}(3) &   b_{2}^{(6)}(3) &   b_{3}^{(6)}(3) \\
\end{pmatrix} =
\begin{pmatrix}
21 & 0 & 0 & 0\\
266 &  14 & 0 & 0\\
1344 &  294 & 0 & 0 \\
-22624 & 2604 &   28 & 0\\
53760 &  -10080 & 420  & 0\\
-32768 & 7168 &  -448 &  8 
\end{pmatrix}.
\end{equation*}
   \section{Study concerning the formula \eqref{eq 1.5} }
  \label{sec 5}  
Obviously we have : \\
       \begin{equation*}
  \displaystyle  \sum_{i=0}^{i=k} \Gamma_{i}^{\left[2\atop{\vdots \atop 2}\right]\times k} =  2^{n(k+1)} \quad \text{if} \quad k< 2n.
   \end{equation*}
   \section{Study concerning the formula \eqref{eq 1.6} }
  \label{sec 6}  
From Lemma 3.3 in Cherly[15] we obtain :\\

\begin{equation}
\label{eq 6.1}
  \displaystyle  a_{i-1}^{(i)}(k)
  = \begin{cases}
  2\cdot2^{k}-25  & \text{if  } i = 2,  \\  
  7\cdot2^{k}-133  & \text{if  } i = 3,  \\ 
 \frac{35}{2} \cdot 2^{k}-605  & \text{if  } i = 4,  \\  
 \frac{155}{4} \cdot 2^{k}-2573  & \text{if  } i = 5,  \\  
 \frac{651}{8} \cdot 2^{k}-10605  & \text{if  } i = 6,  \\  
  \frac{2667}{16} \cdot 2^{k}-43053  & \text{if  } i = 7.
   \end{cases}
\end{equation}
In the following Lemma we compute explicitly $$a_{i-1}^{(i)}(k)$$ for all i and $k > i.$

 \begin{lem}
 \label{lem 6.1}
 We postulate that :\\
 \begin{align}
 \label{eq 6.2}
  \displaystyle a_{i-1}^{(i)}(k)= a(i)\cdot2^{k}+b(i)=  \frac{ \frac{2^{2i-1}-3\cdot2^{i-1}+1}{3}}{2^{i-3}}\cdot2^{k} - \frac{2^{2i+3}-15\cdot2^{i}+7}{3}
 \end{align}
 \end{lem}
 \begin{proof}
 Consider \eqref{eq 6.1}.\\
 Set : \\
\begin{equation*}
  \displaystyle a_{j}= 
  \begin{cases}
  1 & \text{if  } j = 2,  \\  
  7  & \text{if  } j = 3,  \\ 
 35 & \text{if  } j = 4,  \\  
 155 & \text{if  } j = 5,  \\  
 651 & \text{if  } j = 6,  \\  
2667   & \text{if  } j = 7.
     \end{cases} \quad  b_{j}= 
  \begin{cases}
  25 & \text{if  } j = 2,  \\  
  133  & \text{if  } j = 3,  \\ 
 605 & \text{if  } j = 4,  \\  
 2573 & \text{if  } j = 5,  \\  
 10605 & \text{if  } j = 6,  \\  
43053   & \text{if  } j = 7.
     \end{cases}
\end{equation*}
We observe that  the sequence $\{a_{j} \mid j\geqslant 3\} $ satisfy the following recurrence relation $ a_{j}= 4\cdot a_{j-1}+(2^{j-1}-1) $
with the initial condition $ a_{2} =1.$\\
Using successively the above recurrent relation we obtain :\\
\begin{align*}
 a_{j} & = 4\cdot a_{j-1}+(2^{j-1}-1) \\
 4\cdot a_{j-1} & = 4^2\cdot a_{j-2}+4\cdot(2^{j-2}-1) \\
  4^2\cdot a_{j-2} & = 4^3\cdot a_{j-3}+4^2\cdot(2^{j-3}-1) \\
  \vdots \\
    4^{j-4}\cdot a_{4} & = 4^{j-3}\cdot a_{3}+4^{j-4}\cdot(2^3-1) \\
      4^{j-3}\cdot a_{3} & = 4^{j-2}\cdot a_{2}+4^{j-3}\cdot(2^2-1) 
\end{align*}
Summing the above equations we obtain :\\
$$ a_{j} =2^{2j-4}\cdot a_{2} +\sum_{l=0}^{l=j-3}2^{2l}\cdot(2^{j-1-l}-1) = \frac{2^{2j-1}-3\cdot2^{j-1}+1}{3}.$$\\
Equally we observe that  the sequence  $\{b_{j} \mid j\geqslant 3\} $ satisfy the following recurrence relation $   b_{j}= 4\cdot b_{j-1}+33+40\cdot(2^{j-3}-1) $
with the initial condition $ b_{2} =25.$\\
Using as before successively the above recurrent relation we obtain :\\
\begin{align*}
 b_{j} & = 4\cdot b_{j-1}+33+40\cdot(2^{j-3}-1) \\
  4\cdot b_{j-1} & = 4^2\cdot b_{j-2}+33\cdot4+ 40\cdot 4\cdot(2^{j-4}-1) \\
  4^2\cdot b_{j-2} & = 4^3\cdot b_{j-3}+33\cdot4^2+ 40\cdot 4^2\cdot(2^{j-5}-1) \\
   \vdots \\
  4^{j-4}\cdot b_{4} & = 4^{j-3}\cdot b_{3}+33\cdot4^{j-4}+40\cdot4^{j-4}\cdot(2^1-1) \\
 4^{j-3}\cdot b_{3} & = 4^{j-2}\cdot b_{2}++33\cdot4^{j-3}+40\cdot4^{j-3}\cdot(2^0-1) 
\end{align*}

Summing the above equations we obtain :\\

$$ b_{j} = 2^{2j-4}\cdot b_{2}+  33\cdot\sum_{l=0}^{l=j-3}4^{l} +40\cdot\sum_{l=0}^{l=j-3}2^{2l}\cdot(2^{j-3-l}-1) =\frac{2^{2j+3}-15\cdot2^{j}+7}{3}.$$
 \end{proof}
\begin{example}

\begin{equation}
\label{eq 6.3}
a_{7}^{(8)}(k)=  \frac{ \frac{2^{15}-3\cdot2^{7}+1}{3}}{2^{5}}\cdot2^{k} - \frac{2^{19}-15\cdot2^{8}+7}{3}=\frac{10795}{32}\cdot 2^{k}-173485.
\end{equation}

\end{example} 
 
  \section{Computation of  $ \Gamma_{8}^{\left[2\atop{\vdots \atop 2}\right]\times k}$ using the formula \eqref{eq 1.4} }
  \label{sec 7}  
We recall (see section \ref{sec 1} ) that $ \Gamma_{8}^{\left[2\atop{\vdots \atop 2}\right]\times k}$ denotes the number of rank 8
n-times persymmetric matrices over  $\mathbf{F}_2 $  of the form \eqref{eq 1.1}.\\
We shall need the following Lemma  : \\
  \begin{lem}
\label{lem 7.1}
\begin{equation}
\label{eq 7.1}
   \Gamma_{8}^{\left[2\atop{\vdots \atop 2}\right]\times k}
   =   \begin{cases}
 0 & \text{if  } n = 0,  \\  
0 & \text{if  } n = 1,   \\
0 & \text{if  } n = 2,\\
0 & \text{if  } n = 3,\\
16\cdot\big[2^{4k}-240\cdot 2^{3k}+17920\cdot2^{2k}-491520\cdot2^{k}+2^{22}\big] & \text{if   } n=4,\\
496\cdot\big[2^{4k}+9525\cdot2^{3k}-2169440\cdot2^{2k}\\+2^{11}\cdot 68115\cdot2^{k}-9749\cdot2^{18}\big] & \text{if   }  n = 5, \\
10416\cdot\big[2^{4k}+29055\cdot2^{3k}+52983280\cdot2^{2k}\\-2^{10}\cdot 10751745\cdot2^{k}+7323814\cdot2^{16}\big]  & \text{if   }  n = 6.
 \end{cases}
   \end{equation} 
    \begin{equation}
\label{eq 7.2}
   \Gamma_{8}^{\left[2\atop{\vdots \atop 2}\right]\times k}=   \begin{cases}
   511\cdot2^{8n} -765\cdot 2^{7n}  -127762\cdot2^{6n}  +440496\cdot2^{5n}+ 8456800\cdot 2^{4n} \\ 
-57511680 \cdot 2^{3n}+118013952\cdot 2^{2n}-83951616\cdot 2^{n} +14680064  & \text{if   } k=9,\\
     511\cdot2^{8n} +171955\cdot2^{7n} -897890\cdot2^{6n}  -38376240\cdot2^{5n}\\
     + 323250144\cdot 2^{4n} + 271514880\cdot 2^{3n}-436135\cdot2^{14} \cdot 2^{2n}\\
      + 242795\cdot2^{16}\cdot2^{n} -4445\cdot2^{21}   & \text{if   } k=10.
   \end{cases}
   \end{equation} 
\end{lem}
\begin{proof}
Lemma \ref{lem 7.1} follows from Cherly [12,13,14,15 and 16 ]
\end{proof}
  \begin{lem}
  \label{lem 7.2}
We postulate that :\\
\begin{align*}
  \displaystyle  \Gamma_{8}^{\left[2\atop{\vdots \atop 2}\right]\times k}  =
   511 \cdot 2^{8n} + \big[ \frac{10795}{32}\cdot 2^k-173485\big]   \cdot2^{7n}  \\
    \displaystyle   + \big[  \frac{3937}{2^7}\cdot2^{2k}-\frac{1559941}{2^5}\cdot2^k+16768318 \big]    \cdot 2^{6n}\\ 
      \displaystyle  +\big[  \frac{31}{2^6}\cdot2^{3k}-\frac{261919}{2^7}\cdot2^{2k}+\frac{34854113}{16}\cdot2^{k}-643492720 \big]  \cdot 2^{5n} \\
      \displaystyle +\big[ \frac{1}{315\cdot 2^2}\cdot 2^{4k} -\frac{20437}{21\cdot 2^6}\cdot 2^{3k} +\frac{25012451}{9\cdot2^6}\cdot2^{2k} \\
-\frac{3341482313}{84}\cdot2^{k}+7289277664+\frac{35464937\cdot2^{15}}{315} \big]  \cdot2^{4n} \\
  \displaystyle  +\big[    -\frac{1}{21\cdot2^2}\cdot2^{4k}+\frac{102825}{21\cdot2^5}\cdot2^{3k}-\frac{126707455}{21\cdot2^4}\cdot2^{2k} \\ +(110225510+\frac{5819785}{7}\cdot2^{8})\cdot2^{k}-\frac{7081677751\cdot2^{8}}{21}  \big]  \cdot2^{3n} \\
  \displaystyle  +\big[    \frac{1}{18}\cdot2^{4k}-\frac{14735}{24}\cdot2^{3k}+\frac{25506523}{18}\cdot2^{2k}\\
   -\frac{55123739\cdot2^6}{3}\cdot2^{k}  +\frac{169923845\cdot2^{14}}{9}  \big]  \cdot 2^{2n} \\
 \displaystyle  +\big[  -\frac{2}{21}\cdot2^{4k}+\frac{20661}{21}\cdot2^{3k}-\frac{6603656}{3}\cdot2^{2k}\\
       +\frac{24591157\cdot2^{9}}{7}\cdot2^{k}-\frac{150434993\cdot2^{16}}{21}  \big]  \cdot 2^{n}\\
   \displaystyle  +   \frac{16}{315}\cdot\big[2^{4k}-10005\cdot2^{3k}+22047760\cdot2^{2k}\\
 -17459355\cdot2^{10}\cdot2^{k}+35464937\cdot2^{17}\big]\\
  \displaystyle   = \frac{2^{4k}}{1260}\cdot\big[2^{4n}-15\cdot 2^{3n}+70 \cdot 2^{2n}-120\cdot 2^{n}+64\big] \\
  \displaystyle   + \frac{2^{3k}}{21\cdot2^{6}}\cdot\big[2^{4n}-15\cdot 2^{3n}+70 \cdot 2^{2n}-120\cdot 2^{n}+64\big] \cdot\big[651\cdot2^{n}-10672\big]\\
  \displaystyle   + \frac{2^{2k}}{315\cdot2^{7}}\cdot\big[2^{4n}-15\cdot 2^{3n}+70 \cdot 2^{2n}-120\cdot 2^{n}+64\big]\\
    \cdot\big[1240155\cdot2^{2n}-63902160\cdot2^{n}+22047760\cdot2^{5}\big]\\
   \displaystyle + \frac{2^{k}}{21\cdot2^{5}}\cdot\big[2^{4n}-15\cdot 2^{3n}+70 \cdot 2^{2n}-120\cdot 2^{n}+64\big] \cdot\big[226695\cdot2^{3n}
      -29358336\cdot2^{2n}\\
      +1007629056\cdot2^{n}-1163957\cdot2^{13}\big]\\
    \displaystyle  + \frac{1}{315}\big[2^{4n}-15\cdot 2^{3n}+70 \cdot 2^{2n}-120\cdot 2^{n}+64\big] \cdot \big[160965\cdot 2^{4n} -52233300\cdot2^{3n}\\+4487253120\cdot2^{2n}-131715763200\cdot2^{n}+1162115055616  \big]
     \quad \text{for} \quad k\geqslant 9. \\
 \end{align*}
\end{lem}
\begin{proof}
From the formula \eqref{eq 1.4} with i=8 we get \\
\begin{align}
\label{eq 7.3}
\Gamma_{8}^{\left[2\atop{\vdots \atop 2}\right]\times k}
 =  511\cdot 2^{8n} + a_{0}^{(8)}(k)+a_{1}^{(8)}(k)\cdot2^{n}+a_{2}^{(8)}(k)\cdot2^{2n}
 +a_{3}^{(8)}(k)\cdot2^{3n}+a_{4}^{(8)}(k)\cdot2^{4n}\\
 +a_{5}^{(8)}(k)\cdot2^{5n}+a_{6}^{(8)}(k)\cdot2^{6n}+a_{7}^{(8)}(k)\cdot2^{7n}\nonumber \\
=(2^n-1)(2^n-2)(2^n-2^2)(2^n-2^3)\nonumber
\cdot[ 511\cdot2^{4n} + \alpha_{3}^{(8)}(k) \cdot 2^{3n}+ \alpha_{2}^{(8)}(k)\cdot 2^{2n}+ \alpha_{1}^{(8)}(k) \cdot 2^{n}+ \alpha_{0}^{(8)}(k)]. \nonumber
 \end{align}
 From \eqref{eq 7.3} we deduce :\\
  \begin{equation}
  \label{eq 7.4}
   \displaystyle a_{j}^{(8)}(k)= 
   \begin{cases} 
 \displaystyle \alpha_{3}^{(8)}(k)   -7665  & \text{if   } j=7,\\
  \displaystyle   -15\cdot \alpha_{3}^{(8)}(k) + \alpha_{2}^{(8)}(k) +35770  & \text{if   } j=6,\\
   \displaystyle 70\cdot  \alpha_{3}^{(8)}(k) -15\cdot  \alpha_{2}^{(8)}(k) +  \alpha_{1}^{(8)}(k) -61320  & \text{if   } j=5,\\
   \displaystyle  -120\cdot  \alpha_{3}^{(8)}(k) +70\cdot  \alpha_{2}^{(8)}(k) -15\cdot \alpha_{1}^{(8)}(k) + \alpha_{0}^{(8)}(k) +32704  & \text{if   } j=4,\\
    \displaystyle 64\cdot  \alpha_{3}^{(8)}(k) -120\cdot  \alpha_{2}^{(8)}(k) +70\cdot \alpha_{1}^{(8)}(k) -15\cdot \alpha_{0}^{(8)}(k)   & \text{if   } j=3, \\
     \displaystyle 64\cdot  \alpha_{2}^{(8)}(k)  -120 \cdot \alpha_{1}^{(8)}(k)  +70\cdot \alpha_{0}^{(8)}(k) & \text{if   } j=2,\\ 
 \displaystyle  64 \cdot  \alpha_{1}^{(8)}(k)  -120\cdot \alpha_{0}^{(8)}(k)  & \text{if   } j=1,\\ 
 \displaystyle  64\cdot   \alpha_{0}^{(8)}(k)    & \text{if   } j=0.
\end{cases}
    \end{equation}
From \eqref{eq 6.3} and \eqref{eq 7.4} we get :\\
\begin{equation}
\label{eq 7.5} 
\displaystyle  \alpha_{3}^{(8)}(k)  = a_{7}^{(8)}(k)+7665= \frac{10795}{32}\cdot 2^{k}-173485+7665=\frac{10795}{32}\cdot 2^{k}-165820.
\end{equation}
From \eqref{eq 7.3} using \eqref{eq 7.5} we obtain :\\
\begin{align}
\label{eq 7.6}
  \alpha_{2}^{(8)}(k)\cdot 2^{2n}+ \alpha_{1}^{(8)}(k) \cdot 2^{n}+ \alpha_{0}^{(8)}(k)\\
 =\frac{1}{(2^n-1)(2^n-2)(2^n-2^2)(2^n-2^3)}\cdot \Gamma_{8}^{\left[2\atop{\vdots \atop 2}\right]\times k} 
  - 511\cdot2^{4n} - (\frac{10795}{32}\cdot 2^{k}-165820) \cdot 2^{3n}. \nonumber
\end{align}

Combining \eqref{eq 7.6} and \eqref{eq 7.1} with n=4 we deduce :\\
\begin{align}
\label{eq 7.7}
256\cdot \alpha_{2}^{(8)}(k)+16\cdot \alpha_{1}^{(8)}(k)+ \alpha_{0}^{(8)}(k)\\
=\frac{1}{1260}\cdot\big[2^{4k}-240\cdot 2^{3k}+17920\cdot2^{2k}-491520\cdot2^{k}+2^{22}\big] 
  - 2^{7}\cdot 10795\cdot 2^k+645709824.   \nonumber
\end{align}
 Combining \eqref{eq 7.6} and \eqref{eq 7.1} with n=5 we deduce :\\

 \begin{align}
\label{eq 7.8}
1024\cdot \alpha_{2}^{(8)}(k)+32\cdot \alpha_{1}^{(8)}(k)+ \alpha_{0}^{(8)}(k)\\
 =\frac{1}{1260}\cdot\big[2^{4k}+9525\cdot2^{3k}-2169440\cdot2^{2k}+2^{11}\cdot 68115\cdot2^{k}-9749\cdot2^{18}\big] \nonumber  \\    -10795\cdot2^{10}\cdot 2^k+4897767424. \nonumber
\end{align}
 Combining \eqref{eq 7.6} and \eqref{eq 7.1} with n=6 we deduce :\\
  \begin{align}
\label{eq 7.9}
4096\cdot \alpha_{2}^{(8)}(k)+64\cdot \alpha_{1}^{(8)}(k)+ \alpha_{0}^{(8)}(k)\\
\frac{1}{1260}\cdot \big[2^{4k}+29055\cdot2^{3k}+52983280\cdot2^{2k}-2^{10}\cdot 10751745\cdot2^{k}+7323814\cdot2^{16}\big] \nonumber\\
       -10795\cdot2^{13}\cdot2^k+33279\cdot2^{20}. \nonumber
\end{align}

 Using \eqref{eq 7.5}, \eqref{eq 7.7}, \eqref{eq 7.8} and \eqref{eq 7.9} we get :\\
   \begin{equation}
  \label{eq 7.10}
   \displaystyle \alpha_{j}^{(8)}(k) = 
   \begin{cases} 
 \displaystyle  \frac{1}{1260}\cdot \big[2^{4k}-10005\cdot2^{3k}+22047760\cdot2^{2k}-17459355\cdot2^{10}\cdot2^{k}+35464937\cdot2^{17}\big] & \text{if   } j=0,\\
 \\
  \displaystyle  \frac{31}{2^6}\cdot2^{3k}-\frac{12679}{2^3}\cdot2^{2k}+1499448\cdot2^{k}-408345\cdot2^{10}    & \text{if   } j=1,\\
  \\
    \displaystyle \frac{3937}{2^7}\cdot2^{2k}-43688\cdot2^k+222582\cdot2^6  & \text{if   } j=2,\\
    \\
    \displaystyle  \frac{10795}{32}\cdot 2^k-165820 & \text{if   } j=3.\\
   \end{cases}
    \end{equation}
    
  From \eqref{eq 7.4} and  \eqref{eq 7.10} we prove Lemma  \ref{lem 7.2}.  
\end{proof} 
  \begin{example}
 Applying  Lemma  \ref{lem 7.2} with k respectively equal to 9 and 10 we deduce \eqref{eq 7.2}.

\end{example}


\newpage
      \begin{thebibliography}{99}
\bibitem{Landsberg}Landsberg, G {Ueber eine Anzahlbestimmung und eine damit zusammenhangende Reihe},
 {J. reine angew.Math}, {\bf 111}(1893),87-88. 
 \bibitem{Fisher and Alexander}Fisher,S.D and Alexander M.N. {Matrices over a finite field}\\
 {Amer.Math.Monthly 73}(1966), 639-641
  \bibitem{Daykin}  Daykin David E,  {Distribution of Bordered Persymmetric Matrices in a finite field}
  {J. reine angew. Math}, {\bf 203}(1960) ,47-54
   \bibitem{Cherly} Cherly, Jorgen. \\ {Exponential sums and rank of persymmetric  matrices over  $\mathbf{F}_2 $  }\\
{arXiv : 0711.1306, 46 pp} 
  \bibitem{Cherly} Cherly, Jorgen. \\ {Exponential sums and rank of  double  persymmetric  matrices over  $\mathbf{F}_2 $  }\\
{arXiv : 0711.1937, 160 pp} 
  \bibitem{Cherly} Cherly, Jorgen. \\ {Exponential sums and rank of  triple  persymmetric  matrices over  $\mathbf{F}_2 $  }\\
{arXiv : 0803.1398, 233 pp} 
  \bibitem{Cherly} Cherly, Jorgen. \\ {Results about   persymmetric  matrices over  $\mathbf{F}_2 $ and related exponentials sums }\\
   {arXiv : 0803.2412v2, 32 pp} 
   \bibitem{Cherly} Cherly, Jorgen. \\ {Polynomial equations and rank of  matrices over  $\mathbf{F}_2 $ related  to persymmetric matrices}\\
   {arXiv : 0909.0438v1, 33 pp}   
      \bibitem{Cherly} Cherly, Jorgen. \\ {On a conjecture regarding enumeration of n-times persymmetric matrices over $\mathbf{F}_2 $ by rank}\\
   {arXiv : 0909.4030, 21 pp}   
      \bibitem{Cherly} Cherly, Jorgen. \\ { On a conjecture concerning the fraction  of  invertible  m-times  Persymmetric  Matrices over $\mathbb{F}_{2} $}\\ 
    {arXiv : 1008.4048v1, 11 pp}  
    \bibitem{Cherly} Cherly, Jorgen. \\ {Enumeration of some particular n-times  Persymmetric  Matrices over $\mathbb{F}_{2} $ by rank}\\ 
    {arXiv : 1101.2097v1, 18 pp}  
       \bibitem{Cherly} Cherly, Jorgen. \\ {Enumeration of some particular quadruple Persymmetric  Matrices over $\mathbb{F}_{2} $ by rank}\\ 
    {arXiv : 1106.2691v1, 21 pp}  
       \bibitem{Cherly} Cherly, Jorgen. \\ {Enumeration of some particular quintuple Persymmetric  Matrices over $\mathbb{F}_{2} $ by rank}\\ 
    {arXiv : 1109.3623v1, 23 pp}   
        \bibitem{Cherly} Cherly, Jorgen. \\ {Enumeration of some particular 2Nx9 N-Times Persymmetric  Matrices over $\mathbb{F}_{2} $ by rank}\\ 
    {arXiv : 1204.3274v1, 16 pp}          
         \bibitem{Cherly} Cherly, Jorgen. \\ {Enumeration of some particular 2Nx10 N-Times Persymmetric  Matrices over $\mathbb{F}_{2} $ by rank}\\ 
    {arXiv : 1205.6056v1, 18 pp}   
         \bibitem{Cherly} Cherly, Jorgen. \\ {Enumeration of some particular sextuple Persymmetric  Matrices over $\mathbb{F}_{2} $ by rank}\\ 
    {arXiv : 1206.4951v1, 17 pp}   
      \end{thebibliography}
\end{document}